\documentclass[12pt,a4paper,reqno]{amsart}

\usepackage{amsmath,amssymb,amscd,mathrsfs,mathtools,amsthm}
\usepackage[dvipsnames,svgnames,table]{xcolor}
\usepackage{tikz} 
\usepackage{blkarray}
\usepackage{url}
\usepackage{graphicx}
\usepackage{setspace}  
\usepackage{ytableau}
\ytableausetup{centertableaux,boxsize = 1.3em}
\usepackage{tikz-cd}

\usepackage{cases}
\usepackage{pifont} 

\usepackage{nicematrix,tikz}
\usepackage{subcaption}
\usepackage{bm}
 \usepackage{booktabs} 
\definecolor{blond}{rgb}{0.92, 0.89, 0.80}
\definecolor{champagne}{rgb}{0.97, 0.91, 0.81}

\definecolor{GreenDef}{RGB}{27, 132, 5}
\definecolor{BlueDef}{RGB}{11, 78, 188}

\definecolor{RedDef}{RGB}{15,77,146}
\definecolor{Yblue}{RGB}{160, 4, 23}

\definecolor{ThmDef}{RGB}{160, 4, 23}

\usepackage{titletoc}
\usepackage{titlesec}

\titleformat{\section}[hang]
{\scshape\small\color{RedDef}\filcenter}{\S\ \thesection.}{0.2em}{}

\titleformat{\subsection}[runin]
{\scshape\color{RedDef}}{\thesubsection.}{0.2em}{}[.]

\contentsmargin{2.55em}
\dottedcontents{section}[3.8em]{}{1.5em}{1pc}
\dottedcontents{subsection}[6.1em]{}{2.5em}{1pc}
\dottedcontents{subsubsection}[8.1em]{}{3.2em}{1pc}

\usepackage{color}

\usepackage[breaklinks,colorlinks,pagebackref]{hyperref} 

\definecolor{Red}{RGB}{160, 4, 23}
\definecolor{Green}{RGB}{27, 132, 5}
\definecolor{Blue}{RGB}{11, 78, 188}

\hypersetup{linkcolor=Red,citecolor=Green}

\usepackage[capitalise,nameinlink]{cleveref}
\crefformat{equation}{(#2#1#3)}

\newcommand{\GL}{\mathrm{GL}}

\newcommand{\Lift}{\mathrm{Lift}}
\newcommand{\Comp}{\mathrm{Comp}}

\theoremstyle{plain}

\newtheorem{theorem}{Theorem}[section]
\newtheorem{lemma}[theorem]{Lemma}
\newtheorem{proposition}[theorem]{Proposition}

\theoremstyle{definition}

\newtheorem{definition}[theorem]{Definition}
\newtheorem{remark}[theorem]{Remark}
\newtheorem{example}[theorem]{Example}

\newcommand{\marb}{\mathsf{Mar}^b}
\newcommand{\marw}{\mathsf{Mar}^w}
\newcommand{\mar}{\mathsf{Mar}^{bw}}
 
\usepackage{dsfont}

\usepackage[top=2cm,bottom=2cm,right=1.9cm,left=1.9cm]{geometry}

\newcommand{\psum}{\mathsf{sum}}

\usepackage{multicol}

\usepackage{xcolor}

\newcommand{\E}{\mathcal{E}}

\newcommand{\Z}{\mathcal{Z}}

\newcommand{\B}{\mathbf{B}}
\renewcommand{\L}{\mathcal{L}}

\newcommand{\U}{\mathcal{U}}

\newcommand{\best}{\textrm{best}}
\newcommand{\worst}{\textrm{worst}}

\newlength{\originalbase}
\newcommand{\RR}{{\mathbb{R}}}

\newcommand{\e}{\mathsf{e}}
\renewcommand{\k}{\mathsf{k}}

\newcommand{\V}{\mathcal{V}}

\renewcommand{\P}{\mathbb{P}}

 \renewcommand{\P}{\mathbb{P}}
\newcommand{\BW}{\mathbf{BW}}
\newcommand{\one}{\mathbf{1}}

\newcommand{\chb}{\mathsf{Choice}^b}
\newcommand{\chw}{\mathsf{Choice}^w}
\newcommand{\ch}{\mathsf{Choice}^{bw}}

\newcommand{\pup}{\delta^{\uparrow}}
\newcommand{\pdn}{\delta^{\downarrow}}

\title{On the Dimension of the Best--Worst Choice Polytope}

\author{Keivan Mallahi-Karai}
\address{Constructor University, School of Science, Campus Ring I, 28759 Bremen}
\email{kmallahikarai@constructor.university}

\author{Majid Salamat}
\address{Constructor University, School of Science, Campus Ring I, 28759 Bremen}
\email{msalamat@constructor.university}

\begin{document}

\keywords{Best-choice probabilities, best-worst choice polytope, random utility model, multiple-choice polytope}

\begingroup
\def\uppercasenonmath#1{}
\let\MakeUppercase\relax
\maketitle
\endgroup

{
\footnotesize
\hypersetup{linkcolor=RedDef}
}

 
\begin{abstract}
For a set of $n$ alternatives, the best--worst choice polytope is the convex hull of the deterministic best--worst choice vectors induced by all linear rankings of the alternatives. The question of determining the dimension of this polytope as 
a function of $n$ has been raised by Doignon \cite{Doignon2023}. In this paper we answer this question by finding an explicit formula for this dimension. Our proof is motivated by a representation-theoretic interpretation of the relevant rank problem for the symmetric group, but is presented entirely in elementary linear-algebraic terms. The connection with representation theory will be explored elsewhere.
\end{abstract}


\allowdisplaybreaks

\tableofcontents

\section{Introduction}\label{sec:intro}

Random utility models encode stochastic choice by placing a probability distribution on rankings of a finite set of alternatives. This point of view goes back at least to Block and Marschak \cite{BlockMarschak1960}. For ordinary best-choice probabilities, Falmagne \cite{Falmagne1978} obtained a fundamental representation theorem in terms of the non-negativity of the Block--Marschak polynomials. The geometry of the corresponding prediction range, often called the multiple-choice polytope, has subsequently been studied from several points of view.

The best--worst paradigm records more information: from every offered set one observes simultaneously the most preferred and the least preferred alternative. Probabilistic models for this setting were developed, among others, by Marley and Louviere \cite{MarleyLouviere2005}. Colonius \cite{Colonius2021} introduced best--worst analogues of the Block--Marschak polynomials and investigated the associated random utility representation problem. Doignon \cite{Doignon2023} subsequently showed, in particular, that the best--worst Block--Marschak inequalities are not sufficient once four or more alternatives are present; see also the later technical note of Colonius \cite{Colonius2024}. From the polyhedral point of view, Doignon gave a complete analysis for four alternatives. In particular, he showed that the corresponding polytope has $24$ vertices and  dimension $22$, and obtained a complete facet description in that case. An alternative approach to this problem is presented in \cite{CMKS2024}. 

The purpose of this paper is to establish a closed formula for the dimension of best-worst choice polytope for an arbitrary number of alternatives. Our proof is inspired by ideas related to the representation theory of symmetric groups, it is however presented in an elementary fashion. We only use basic tools from linear algebra and combinatorics. It is also worthwhile to mention that best--worst Block--Marschak inequalities do not appear to play an important role here. 

Let us start by setting some basic notation.  Let $n \ge 3$ be an integer, write $[n]:=\{1,\dots,n\}$, and denote by $\Pi_n$ the set of all rankings of $[n]$. More generally, if $A\subseteq[n]$, we denote by $\Pi_A$ the set of all words in which every element of $A$ occurs exactly once; in particular, $\Pi_\emptyset$ consists of the empty word. A ranking $\pi=(\pi(1),\dots,\pi(n))\in\Pi_n$ can be viewed as a total order $\prec_\pi$ on $[n]$ defined by
$$ \pi(1)\prec_\pi\cdots\prec_\pi\pi(n). $$ We write $a\preceq_\pi b$ if either $a=b$ or $a\prec_\pi b$, and say that $a$ precedes $b$ in $\pi$. 

\medskip
A probability vector on $\Pi_n$ is a function $p:\Pi_n\to[0,1]$ satisfying $\sum_{\pi\in\Pi_n}p(\pi)=1$. We write $p_\pi:=p(\pi)$ and denote the set of all probability vectors on $\Pi_n$ by $\Delta_n$. Each element of $\Delta_n$ can be viewed as an element of the real vector space $\RR^{\Pi_n}$. Thus the set of all such probability vectors 
$$ \Delta_n:=\left\{p=(p_\pi)_{\pi\in\Pi_n}:\sum_{\pi\in\Pi_n}p_\pi=1,\quad p_\pi\ge 0\text{ for all }\pi\in\Pi_n\right\} $$
define a simplex in $\RR^{\Pi_n}$ whose vertices are elements of $\Pi_n$. 

An event $E$ is simply a subset of $\Pi_n$. Given a probability vector $p\in\Delta_n$, the probability of $E$ with respect to $p$ is defined by
$ \P^p(E)=\sum_{\pi\in E}p_\pi. $

\medskip
\paragraph{\it Best-Worst-choice probabilities.} Let
\[ \Theta_n= \{(a,b,B): B\subseteq[n],\ |B|\geq 2,\ a,b\in B,\ a\neq b\} \]
denote the set of all subsets $B$ of $[n]$ with a choice of distinct elements $a, b \in B$. 

For $(a,b,B) \in\Theta_n$, consider the probability under $p \in \Delta_n$ of the event that $a$ is the best element and $b$ is the worst element of $B$ for a random permutation $\pi$ sampled according to $p$. In other words, set
\begin{equation}\label{def:pan}
\P^p_{a,b, B}:=\P\big[\{\pi\in\Pi_n:a\preceq_\pi x \preceq_\pi  b\text{ for all } x\in B\}\big]
=\sum_{\substack{\pi\in\Pi_n:\\a\preceq_\pi x \preceq_\pi   b,\ \forall x\in B}}p_\pi.
\end{equation}
For instance, for $n=4$ and $B=\{1,2, 3\}$, we have $\P_{1,2, \{1,2,3\}}=p_{4132}+p_{1432}+p_{1342}+ p_{1324}$. Each $(\P^p_{a,B})_{(a,b,B)\in\Theta_n}$ satisfies the normalization conditions
\begin{equation}\label{normal}
\sum_{a, b \in B, a \neq b }\P_{a,b, B}=1,\qquad\text{for every }B\subseteq[n] \text{ with } |B| \ge 2.
\end{equation}

One thus obtains a map from $\RR{\Pi_n}$ to $\RR^{\Theta_n}$ mapping $p$ to $(\P^p_{a,b, B})_{(a,b,B) \in \Theta_n}$. The image $\Delta_n$ under this map is the so-called best-worst choice polytope in  $\RR^{\Theta_n}$, denoted by $\BW_n$. We determine the dimension of this polytope. 

\medskip

\begin{theorem}\label{thm:main}
For every $n\geq 2$, the dimension of the best--worst choice polytope on $n$ alternatives is given by 
$$ \dim \BW_n = n(n-3)2^{n-2}+\binom{n}{2}-\binom{n}{3}+n.  $$
\end{theorem}

\medskip

\noindent
{\bf Strategy of the proof.}
Let us begin by outlining the proof strategy. The dimension of the Best–Worst polytope is determined by the rank of a matrix whose rows  and columns are indexed, respectively, by $\Theta_n$ and $\Pi_n$. This point of view was used in  \cite{SMK2026} to give another proof for a theorem of Doignon and Saito \cite{DS} for the dimension of the best choice polytope. Rather than working directly with matrices, as in \cite{SMK2026}, it will be more convenient to use the language of functions on the various sets such as $\Pi_n$, $\Theta_n$, etc.

To compute the dimension of the relevant function space, we introduce a filtration by subspaces by subspaces $\V_{\le k}$, where $\V_{\le k}$ is generated by the best-worst choice functions corresponding to sets $B$ with $|B|\leq k$; see \eqref{def:vk}. This yields the nested sequence of subspaces in \eqref{filt} and reduces the problem to computing the dimensions of the successive quotients $\V_{\le k}/\V_{\le k-1}$. Each set $B$ with $|B|=k$ contributes a subspace to this quotient, all of the same dimension. These subspaces, however, need not form a direct sum. Determining the dimension of each individual contribution and analyzing the intersections among them constitute the main part of the proof. In order to facilitate reading the paper, we will first provide a proof along the same lines for the aforementioned theorem of Doignon and Saito \cite{DS}. This proof is technically simpler, but highlights the key steps of the proof of Theorem \ref{thm:main}. 

The principal advantage of working with function spaces rather than directly with row spaces is that the former carry a natural inner product, making arguments involving orthogonality particularly transparent. Nevertheless, the entire proof could, in principle, be reformulated in terms of the row space of the matrix $B$. For example, Proposition \ref{prop:choice-filtration} admits a matrix-theoretic interpretation in terms of the number of new pivots contributed by a given block when the matrix is arranged in an appropriate block form. 

\medskip 

\noindent
{\bf Connection to representation theory.}
The problem of determining the rank associated with the polytope can also be formulated in terms of the representation theory of the symmetric group. Let $S_n$ denote the group of all permutations of $[n]$. Recall that a linear representation of a group $G$ on a vector space $V$ is a group homomorphism $\rho\colon G\to\GL(V)$, where $\GL(V)$ denotes the group of invertible linear transformations of $V$.

In the special case $G=S_n$, the theory dates back to the pioneering work of Frobenius in 1896 \cite{Frobenius}, whose aim was to extend Fourier analysis to the setting of non-abelian groups. The irreducible representations of $S_n$ are naturally indexed by the partitions of $n$. The natural action of $S_n$ on the set of rankings induces a permutation representation, which can be decomposed into a direct sum of irreducible representations. This makes it possible to recast the combinatorial problem of computing the rank as the problem of determining the dimension of an appropriate representations of $S_n$. From this perspective, the proof developed here amounts to an analysis of one such representation. This connection will be explored in a companion paper.

\subsection{Acknowledgements}
This work was supported by the German Research Foundation (DFG) under grant no. MA 6503/1-1. 


\section{Notation and terminology}\label{notation}
In this section, we will introduce some additional notation that will be used in the rest of the paper. 

\medskip

\noindent
{\it Rankings of a finite set and their combinatorics}. 
Let $X$ be a finite set. A ranking of $X$ is a total linear order on $X$.
We write $\Pi_X$ for the set of all rankings of $X$. An element $\pi \in \Pi_X$ will be represented as 
$\pi=(\pi_1,\ldots,\pi_k)$, where $k=|X|$.   Viewed in the setting of dynamic choice theory, one can regard $\pi$ as a ranking of $n$ alternatives. 
Hence, we refer to $\pi_1$ as the {\it best} (or most preferred) and $\pi_k$ the {\it worst} (or the least preferred) alternative. It is clear that if $X$ has $n$ elements, then $\Pi_X$ has $n!$ elements. 

Let $\pi \in \Pi_X$. The restriction of $\pi$ to $B$ is ranking of $B$ obtained by taking the restriction of the partial order associated with $\pi $ to $B$. 
 We denote this ranking by $\pi|_B$. For instance, for $\pi= (2,4,1,3,5) \in \Pi_5$ and $B= \{ 1,2,3 \}$, 
we have $\pi|_B= ( 2,1,3)$.  The {best} element and the {worst} element of  $B$ under $\pi$ are simply $(\pi|_B)_1$ and $(\pi|_B)_k$, where $k=|B|$.  We introduce some basic calculus on rankings that will be used in the sequel. 

\begin{itemize}
\item For $\pi \in \Pi_X$, and $ a \in X$, we write $ \pi \setminus a$ for the element of $\Pi_{X \setminus \{ a \} }$ obtained by dropping $a$ from the ranking. For instance, for $ \pi=(3,4,1,2) \in \Pi_{4}$ we have $ \pi \setminus 4 = (3,1,2)$. 

\item Let $X_1$ and $X_2$ be disjoint sets and $ \pi_1 \in \Pi_{X_1}$ and $\pi_2 \in \Pi_{X_2}$, then we write $ \pi_1 \bullet \pi_2$ for the ranking of $X_1 \cup X_2$ in which all elements of $X_1$ precede all elements of $X_2$. Viewed as a partial order on $X_1 \sqcup X_2$, this corresponds to the lexicographic order in which all elements of $X_1$ precede all elements of $X_2$.  For instance, let $X_1 = \{ 1,2, 3 \}, X_2= \{ 4, 5 \}$. If $\pi_1= (2,3,1)$, $\pi_2= (5,4)$, then $\pi_1 \bullet \pi_2= (2,3,1,5,4)$. 

\item For every $ C \subseteq [n]$, we write $\pup_C$ for the increasing and $\pdn_C$ for the decreasing ranking of $C$. For instance, if $C= \{ 2,5,4 \}$, then 
$\pup_C= (2,4,5)$ and $\pdn_C= ( 5,4,2)$. 

\end{itemize}

\medskip

\noindent
{\it Vector space of functions on a set.}
When $X$ is a finite set, the set of all maps from $X$ to $\RR$ will be denoted by $\RR^X$. It is clear that $\RR^X$ is a vector space of dimension $|X|$, the cardinality of $X$. The indicator function of a subset $A \subseteq X$ maps all elements of $A$ to  $1$ and all elements of $A^c$ to $0$. We denote this function by $\one_A$. The special case $A=X$, this is the constant function $1$ on $X$ in which case we simply write $\one$ instead of $\one_X$.  For each $x \in X$, 
we write $\e_x= \one_{\{ x \} }$ for the function from $X$ to $\RR$ that is $1$ on $x$ and $0$ elsewhere. It is clear that $\{ \e_x \}_{x \in X }$ forms a basis for $ \RR^X$.
We refer to this basis as the canonical basis of $\RR^X$.

\medskip

\noindent{\it Linear structure.}
Let us first recall some basic facts from linear algebra. For a family $\{ W_i: i \in I \}$ of linear 
subspaces of a vector space $W$, we write $\sum_{i \in I}$ for the subspace of $W$ spanned by all $W_i$, $i \in I$. We say that 
$\sum_i W_i$ is a direct sum if every element of $  \sum_i W_i$ has a unique representation as $\sum_i w_i$ with $w_i \in W_i$. This is equivalent to the property that if $\sum_i w_i =0$ with $w_i \in W_i$ then $w_i=0$ for all $i \in I$. We indicate this by writing $W= \bigoplus_{i} W_i$. Note that this implies that
$\dim W= \sum_i \dim W_i$. 

\medskip

\noindent
 We equip the real vector space $\RR^X$ with the inner product 
\begin{equation}\label{def:inner}
\langle f, g \rangle_X:= \sum_{x \in X} f(x) g(x), \quad f, g \in \RR^X. 
\end{equation}
When $X$ is clear from the context we drop the subscript $X$ and denote the inner product by $ \langle \cdot , \cdot \rangle$. 
If $W$ is a linear subspace of $\RR^X$, its orthogonal complement with respect to this inner product consists of vectors $v \in \RR^X$ such that $ \langle v, w \rangle=0$ for 
all $w \in W$. We denote this subspace by $W^\perp$. Note that since $ \langle \, , \, \rangle$ defines a non-degenerate pairing on $V$,  we have $\dim W+ \dim W^\perp = \dim \RR^X= |X|$. The orthogonal complement of the one-dimensional subspace $\RR \one$ of all constant functions on $X$ is easily seen to be 
\[ ( \RR \one )^\perp=  \left\{ f \in \RR^X: \, \sum_{x \in X} f(x)=0 \right\}. \]

\begin{remark}\label{sumzero}
Let $W$ be the subspace spanned by  $S= \{ \e_x- \e_y: x, y \in X \} $. If $f \in W^\perp$ then 
$ \langle f, \e_x - \e_y \rangle_X=0$, implying that $f(x)=f(y)$  for all $x,y \in X$. Hence $f$ is a constant function, and 
$W^\perp= \RR \one$. This shows that $S$ spans $\RR^X_0$. 
\end{remark}

\medskip

\noindent
{\it Choice polytopes}. We will now define  the polytopes that will be studied in this paper. For each $n \ge 1$, set 
\begin{equation}
\begin{split}      
\Theta_n^{\B} &= \{(a,B): B\subseteq[n],\ |B|\geq 1,\ a\in B \}, \\
\Theta_n^{\BW} &= \{(a,b,B): B\subseteq[n],\ |B|\geq 2,\ a,b\in B,\ a\neq b\}.
\end{split}
\end{equation}
Using this notation, we can now define the following elements of $\RR^{\Pi_n}$:
\begin{equation}\label{def:fab}
\begin{split}      
 \widetilde{ f}^B_{a}(\pi) &:=\one\{a\text{ is best in }B \} = \{ \pi \in \Pi_n: (\pi|_B)_1= a \}, \\
 \widetilde{ f}^B_{a,b}(\pi) &:=\one\{a\text{ is best in }B\text{ and }b\text{ is worst in }B\}=
\{ \pi \in \Pi_n: (\pi|_B)_1= a, \, (\pi|_B)_{|B|}=b \}, 
\end{split}
\end{equation} 
where in the first definition we assume that $(a,B) \in \Theta_n^\B$ and in the second definition we have $(a,b,B) \in 
\Theta_n^\BW$.  For $\pi\in\Pi_n$, the associated deterministic best vector and the  best--worst vector are defined by 
\[ v_\pi^\B = \bigl( \widetilde{f }^B_{a}(\pi)\bigr)_{(a,B)\in\Theta_n^{\B}} \in \RR^{\Theta_n^{\B}}, \qquad
v_\pi^\BW = \bigl(\widetilde{ f}^B_{a,b}(\pi)\bigr)_{(a,b,B)\in\Theta_n^{\BW}} \in \RR^{\Theta_n^{\BW}}. \]

\begin{definition}
The \emph{best choice polytope} and the \emph{best--worst choice polytope} on $n$ alternatives are defined by 
\[ \B_n:= \operatorname{conv}\{v^\B_\pi:\pi\in\Pi_n\}, \qquad \BW_n=\operatorname{conv}\{v^\BW_\pi:\pi\in\Pi_n\},  \]
Here, the notation $ \operatorname{conv}(A)$ denotes the convex hull of the subset $A$ of a real vector space. 
\end{definition}

In other words, $\B_n$ and $\BW_n$ are polytopes whose set of extreme points are respectively given by $v_\pi^\B$ and $ v_\pi^\BW$ as
$\pi$ is allowed to vary in the set of rankings $\Pi_n$. 
Now, if $p=(p_\pi)_{\pi\in\Pi_n}$ is a probability distribution on rankings, and $\P^p$ denotes the corresponding probability measure on $[n]$ then the induced best and best--worst probabilities are given by 
\begin{equation}
\begin{split}      
\P_p\bigl(a\text{ is best in }B\bigr) &= \sum_{\pi\in\Pi_n}p_\pi f^B_{a}(\pi), \\
\P_p\bigl(a\text{ is best in }B,\ b\text{ is worst in }B\bigr) &= \sum_{\pi\in\Pi_n}p_\pi f^B_{a,b}(\pi). 
\end{split}
\end{equation}

\section{Some function spaces and linear operators}
In this section we will define some vector spaces of functions on various sets and prove some basic properties about them. These properties will be used in the proof of the main theorem. 

\medskip

Let $B$ be a set.  We write $ B^{(2)}$ for the set of pairs of distinct points of $B$, that is, 
$ B^{(2)}= \{ (a, b): a, b \in B, a \neq b \}$. Let $k= |B|$. It is clear that $| B^{(2)}| = k( k-1)$. We now define three {\it marginal} operators. 
\begin{equation}\label{maginaldef}
\begin{split}      
\marb: \RR^{\Pi_B} \to \RR^B, \quad & \marb(h)(a)= \sum_{ \pi_1= a} h(\pi), \\
\marw: \RR^{\Pi_B} \to \RR^B, \quad & \marw(h)(b)= \sum_{ \pi_k= b} h(\pi), \\
\mar: \RR^{\Pi_B} \to \RR^{B^{(2)}} , \quad & \mar(h)(a,b)= \sum_{ \pi_1= a, \pi_k=b} h(\pi),
\end{split}
\end{equation}

\begin{example}\label{exm1}
Let $B=\{1,2,3\}$, and define $h:\Pi_B\to\RR$ by
$$
h(123)=1,\qquad h(132)=2,\qquad h(213)=-1,
$$
$$
h(231)=3,\qquad h(312)=0,\qquad h(321)=-2.
$$
Then $\marb$ and $\marw$ represented as column vectors are given by 
$$
\marb(h)
=
\begin{pmatrix}
h(123)+h(132)\\
h(213)+h(231)\\
h(312)+h(321)
\end{pmatrix}
=
\begin{pmatrix}
3\\
2\\
-2
\end{pmatrix},
$$
whereas
$$
\marw(h)
=
\begin{pmatrix}
h(231)+h(321)\\
h(132)+h(312)\\
h(123)+h(213)
\end{pmatrix}
=
\begin{pmatrix}
1\\
2\\
0
\end{pmatrix}.
$$
Note that for rankings of a set with three elements, the best and the worst elements already determine the ranking uniquely. So, for instance, we have 
$\mar(h)(1,2)= h(1,3,2)= 2.$
It is sometimes useful to regard $\mar(h)$ as a matrix with zero diagonal entries. In this example we have 
$$
\mar(h)
=
\begin{pmatrix}
0&2&1\\
3&0&-1\\
-2&0&0
\end{pmatrix}.
$$
\end{example}

\begin{remark}\label{psums}
Note that in the above example $\marb(h)$ and $\marw(h)$ can be recovered as the row sum and column sum vectors of the matrix $\mar(h)$. This simple fact will be used later. 
\end{remark}

We will also define three {\it choice} operators:
\begin{equation}
\begin{split}      
\chb:  \RR^B \to \RR^{\Pi_B}, \quad & \chb(u)(\pi)=  u(\pi_1), \\
\chw:  \RR^B \to  \RR^{\Pi_B}, \quad & \chw(u)(\pi)= u(\pi_k), \\
\ch: \RR^{B^{(2)}} \to \RR^{\Pi_B} , \quad & \ch(u)(\pi)= u( \pi_1, \pi_k). 
\end{split}
\end{equation}

\begin{example}\label{exm2}
Let $B=\{1,2,3\}$ and let $u\in\RR^B$ be given by setting  $$ u(1)=1, \quad u(2)=-2, \quad u(3)=3.$$
Here are some examples of application of the  choice operators on $u$: 
$$
\chb(u)(123)=u(1)=1,\qquad
\chb(u)(231)=u(2)=-2,\qquad \chw(u)(312)=u(2)=-2.
$$

Now let $v\in\RR^{B^{(2)}}$ be given by
$$
v(1,2)=2,\qquad v(1,3)=1,\qquad
v(2,1)=3,\qquad v(2,3)=v(3,1)=v(3,2)=0.
$$
Then, we have
$$
\ch(v)(123)=v(1,3)=1,\qquad
\ch(v)(231)=v(2,1)=3. 
$$
\end{example}

\medskip

The following simple lemma shows that the choice and marginal operators are adjoint to each other. This property will be used later.

\begin{lemma}[Adjoint property]\label{adjoint}
Suppose $h \in \RR^{\Pi_B}$. 
\begin{enumerate} 
\item For $u \in \RR^B$ we have  $\langle \marb(h), u \rangle_B= \langle h, \chb(u) \rangle_{\Pi_B}.$
\item For $u \in \RR^{B^{(2)}}$ we have $\langle \mar(h), u \rangle_{ B^{(2)}} = \langle h, \ch(u) \rangle_{\Pi_B}. $
\end{enumerate}
\end{lemma}

\begin{proof}
For the first assertion, we compute
$$ \langle \marb(h),u\rangle_B = \sum_{a\in B} \left( \sum_{\substack{\pi\in\Pi_B\\ \pi_1=a}}h(\pi) \right)u(a) = \sum_{\pi\in\Pi_B}h(\pi)u(\pi_1)= \langle h,\chb(u)\rangle_{\Pi_B}. $$
The proof of the second statement is similar.  Let $u\in\RR^{B^{(2)}}$. Then
$$ \langle \mar(h),u\rangle_{ B^{(2)}} = \sum_{(a,b)\in B^{(2)}}\mar(h)(a,b)u(a,b) = \sum_{(a,b)\in B^{(2)}} \left( \sum_{\substack{\pi\in\Pi_B\\ \pi_1=a,\ \pi_k=b}} h(\pi) \right)u(a,b).$$

Each ranking $\pi\in\Pi_B$ occurs exactly once in this sum, namely for
$(a,b)=(\pi_1,\pi_k)$. Hence $$ \langle \mar(h),u\rangle_{ B^{(2)}} = \sum_{\pi\in\Pi_B} h(\pi)u(\pi_1,\pi_k)=  \langle h,\ch(u)\rangle_{\Pi_B}. $$

\end{proof}

We are primarily interested in the images of the choice operators introduced above. The point of Lemma \ref{adjoint} is that it allows us to translate questions about these images into corresponding questions about the marginal operators. The latter turn out to be easier to analyze.

\subsection{Some function spaces}\label{sec:lift}
Let $B$ be a set and $C\subseteq B$. There is a natural map 
$$\Lift^B_C: \RR^{\Pi_C} \to \RR^{\Pi_B}, \qquad \Lift^B_C(h)(\pi)=h( \pi|_C). $$ 
It is clear that $\Lift^B_C$ is an injective linear map. The image of $\Lift^B_C$ is the linear subspace of $ \RR^{\Pi_B}$ consisting of 
those functions on rankings of $B$ that depend {\it only} on the restriction of the ranking to $C$. We write 
\[  \mathcal U^B_C:= \Lift^B_C( \RR^{\Pi_C}). \]
When $B$ is known from context, we simply write $\mathcal U_C$ to denote this subspace. Note that $\U^B_C$ is spanned by the images $  \Lift^B_C( \e_\tau)$ of the canonical basis of $\RR^{\Pi_C}$. A simple computation shows that 
\[  \Lift^B_C( \e_\tau)= \sum_{ \pi|_C= \tau} \e_{\pi}. \]
In the special case that $C= B \setminus \{ i \}$, this simplifies to 
\[ \Lift^B_C( \e_\tau)= \sum_{ \pi \setminus i= \tau} \e_{\pi}. \]
The sum here is over all $|B|$ permutations $\pi$ that can be obtained by inserting $i$ into $\tau$.  
 We also define 
$$ \L_B: =\sum_{C\subsetneq B}\mathcal U^B_C. $$

\begin{example} \vspace{3mm}

Let $B=\{1,2,3\}$ and $C=\{1,2\}$. The set $\Pi_C$ consists of the two rankings $12$ and $21$. Consider the function $h\in\RR^{\Pi_C}$ defined by $$ h(12)=1,\qquad h(21)=0. $$
Then we have 
$$ \Lift^B_C(h)(123) = \Lift^B_C(h)(132) = \Lift^B_C(h)(312) = 1, $$
whereas
$$ \Lift^B_C(h)(213) = \Lift^B_C(h)(231) = \Lift^B_C(h)(321) = 0. $$
Note that the value of $\Lift^B_C(h)$ only depends on the relative order of $1$ and $2$. In particular, it is unaffected by the position of $3$.  Since functions depending on a singleton are constant, we may write
$$ \L_{\{1,2,3\}} =\mathcal U_{\{1,2\}} + \mathcal U_{\{1,3\}} + \mathcal U_{\{2,3\}}. $$
Thus a typical element of $\L_{\{1,2,3\}}$ has the form
$$
f(\pi)
=
h_{12}(\pi|_{\{1,2\}})
+
h_{13}(\pi|_{\{1,3\}})
+
h_{23}(\pi|_{\{2,3\}}),
$$
where $h_{ij}\in\RR^{\Pi_{\{i,j\}}}$.
One can show that  $ \dim \L_{\{1,2,3\}}=4. $ In particular, $\L_{\{1,2,3\}}$ is a proper subspace of $\RR^{\Pi_B}$, which has dimension $6$. 
\end{example}

\begin{proposition}\label{ortho}
Let $B$ be a set with cardinality at least $1$. Then 
\begin{enumerate}
\item For $C_1 \subseteq C_2 \subseteq B$, we have $\U_{C_1}^B \subseteq \U_{C_2}^B$. 
\item If $C$ is non-empty then $\U_C^B$ contains the constant function $\one$.
\item \[ \L_B =\sum_{x \in B}\mathcal U^B_{B \setminus \{ x \} }. \]
\item In particular, $h \in \L_B^\perp$ if for every $i \in B $ and every $ \tau \in \Pi_{ B \setminus \{ i \} }$ we have
\[ \sum_{ \pi \setminus i = \tau } h(\pi)=0. \]
\end{enumerate}
\end{proposition}

\begin{proof}
Parts (1) and (2) are immediate from the definition of $\U_C^B$.  Part (3) follows formally from (1). 
Finally, part (4) follows from  (3) the description of the basis of $\U_{B \setminus \{ i \} }^B$ given above. 
\end{proof}

 \section{The dimension of the choice polytope}
In order to motivate the proof of Theorem \ref{thm:main} that deals with the dimension of $\BW_n$, we will first consider the simpler problem of $\dim \B_n$ and illustrate the approach that will be used later for the more challenging problem of determining $\dim \BW_n$. This approach is clearly not the shortest one, but has the advantage that it generalizes very naturally, although several steps of the proof are somewhat simpler.

We first start by setting some additional notation. 
Fix a nonempty subset $B\subseteq[n]$ and $a\in B$. The functions $f^B_a \in \RR^{\Pi_B}$ defined by
$$ f^B_a(\pi)=\mathbf{1}_{\{a\text{ is the best element of }B\text{ under }\pi\} }= \begin{cases} 1 & \textrm{if } \, \pi_1=a\\
0 & \,  \textrm{otherwise } 
\end{cases}     $$

Let us mention that the function $f^B_a \in \RR^{\Pi_B}$ is formally different from functions $\widetilde{ f}^B_a$ defined in \eqref{def:fab}, which are elements of 
$\RR^{\Pi_n}$.   In fact, one can easily verify that $ \widetilde{f}_a^B = \Lift_{B}^{[n]} f_a^B$.
It will be convenient to first work with these {\it local} functions. When we put them together in Proposition \ref{prop:choice-filtration}, we will go back to 
 $\widetilde{ f}^B_a$ which all belong to the same vector space $\RR^{\Pi_n}$. 

For a subset $B$ as above, set 
$$ \E_B = \operatorname{span}\{f^B_{a}:a \in B \} \subseteq \RR^{\Pi_B} $$
It is clear that $f_a^B= \chb{\e_a}$ and hence $\E_B= \chb(\RR^B) \subseteq  \RR^{\Pi_B}$. It is easy to see that 
$\chb$ is injective and hence  $\dim \E_B=|B|$. Injectivity of $\chb$ implies that its adjoint $ \marb$ is surjective. The next lemma determines the image of $\marb$ when restricted to the orthogonal complement of $ \L_B$.

\begin{lemma}\label{lem:choice-first-marginal}
Let $B$ be a finite set with $|B|\ge2$. Then
$$ \marb(\L_B^\perp) = \RR^B_0. $$
\end{lemma}

\begin{proof}
First note that if $h \in \L_B^\perp$ then it follows from Lemma \ref{adjoint} that 
\[ \langle \marb (h), \one \rangle = \langle h, \chb(\one) \rangle= \langle h, \one \rangle =0,  \]
where the last equality follows from $\one  \in \L_B$. 
For the reverse inclusion, using Remark \ref{sumzero}, it suffices to show that $\e_p-\e_q \in \marb(\L_B^\perp) $ for all distinct
$p, q \in B$. We will construct $h \in \L_B^\perp$ with $\marb(h)=\e_p - \e_q$. 

Define $C_{p<q}$ to be the set of rankings of $B$ in which $p$ appears immediately before $q$, the elements before $p$ are increasing and those after $q$ are in decreasing order. In other words, $C_{p<q}$ consists of rankings of the form 
\[  \pup_{C_1} \bullet (p,q) \bullet \pdn_{C_2} \]
where $C_1 \sqcup C_2=  B \setminus \{ p, q \} $. Define $C_{q<p}$ in a similar way with $(p,q)$ in the middle replaced by $(q,p)$.
Define 
\[ h(\pi)= \begin{cases} (-1)^{|C_1|} & \textrm{if } \, \pi \in C_{p<q}\\
(-1)^{|C_1|+1} & \textrm{if } \, \pi \in C_{q<p}  \\
0 &  \,  \textrm{otherwise } \\
\end{cases}    \]
We first use Proposition \ref{ortho} to show that $h \in \L_B^{\perp}$. Let $i \in B$ and $\tau$ be a ranking of $B \setminus \{ i \}$. We need to show that 
\[ \sum_{ \pi \setminus i = \tau } h(\pi)=0. \]
Let us first assume that $i \not \{ p, q \}$, and observe that if $p$ and $q$ are non-adjacent in $\tau$ they will also non-adjacent over all possible $\pi$ obtained by inserting $i$. Hence for all such $\tau$ all the terms in the sum are zero and there is nothing to prove. Thus, we can assume that $p$ and $q$ are adjacent in $\tau$. First, assume that $p$ precedes $q$ and hence
$\tau= \tau_1 \bullet (p,q) \bullet \tau_2$. Let $C_1$ be the elements of $B$ appearing before $p$ and $C_2$ the elements of $B$ appearing after $q$, so that 
$C_1 \cup C_2= B \setminus  \{ i,p,q \}$. It follows from the definition of $h$ that $h(\pi)=0$ unless $ \tau_1= \pup_{C_1}$ and $\tau_2= \pdn_{C_2}$. Note that then there is a unique way of inserting  $i$ in $C_1$ or in $C_2$ so that the part of the ranking before $p$ and the part after $q$ are increasing and decreasing. As an example, suppose $n=8$, $p=4,q=5$ and $i=3$. Let 
$\tau= (1,2,4,5,8,7,6)$. Here $\tau_1=(1,2)$ and $\tau_2=(8,7,6)$. Now, one can insert $3$ at the end of $\tau_1$ and obtain
$\pi=(1,2,3,4,5,8,7,6)$ or at the end of $\tau_2$ and obtain $\pi=(1,2,4,5,8,7,6,3)$.
 It is clear that for these two permutations $\pi_1$ and $\pi_2$ we have $$h(\pi_1)+ h(\pi_2)= (-1)^{|C_1|}+ (-1)^{|C_1|+1}=0.$$ So the claim holds when $i \neq p,q$. When $i=p$, then by a similar argument we must have $\tau= \tau_1 \bullet q \bullet \tau_2$, where $ \tau_1= \pup_{C_1}$ and $\tau_2= \pdn_{C_2}$. In this case, $q$ can be inserted either before or after $p$ and from the definition of $h$ resulting in rankings $\pi_1, \pi_2 \in \Pi_B$  with $h(\pi_1)+ h( \pi_2)=0$. 

Finally, observe that $\marb(h)(p)= \sum_{\pi_1= p}  h(\pi)=1$ as there is a unique $\pi$ in the support of $h$ starting with $p$. A similar argument shows that 
$ \marb(h)(q)=-1$. For every $a \in B \setminus \{ p,q \}$, if $h(\pi) \neq 0$ and $\pi_1=a$ then, we have $\pi= \tau_1 \bullet (p,q) \bullet \tau_2$ or $\pi= \tau_1 \bullet (q,p) \bullet \tau_2$  with $\tau_1 \in \Pi_{C_1}$ and $a$ the smallest element of $C_1$. Since the value of $h$ on these two subsets of equal size cancel each other, the claim follows in this case as well. 
\end{proof}

\begin{lemma}\label{lem:choice-local}
Let $B\subseteq[n]$ with $|B|\ge 2$. Then
$$ \E_B\cap \L_B=\RR\mathbf{1}. $$
In particular, we have $$ \dim \E_B/(\E_B\cap \L_B)=|B|-1. $$
\end{lemma}

\begin{proof}
It is clear that $ \RR \one \subseteq \E_B \cap \L_B$. To prove the reverse inclusion, let $g\in \E_B \cap \L_B$. Since $g \in \E_B$, we can write  $g = \chb(u)$ for some $u \in \RR^B$. 
 Since $g \in \L_B$, we know that $g$ is orthogonal to every element of $\L_B^\perp$. This implies that for every $h \in \L_B^\perp$ we have
 \[ 0 = \langle g, h \rangle_{\Pi_B}=   \langle \chb(u), h \rangle_{\Pi_B} = \langle u, \marb(h) \rangle_{B}. \]
Using Lemma \ref{lem:choice-first-marginal} we deduce that $u$ is orthogonal to $\RR^B_0$, hence $u$ is constant. This implies that $g$ is constant finishing the proof.  Since $\dim E_B=|B|$, the quotient has dimension $|B|-1$.
\end{proof}

Note that, as $B$ ranges over the subsets of $[n]$, the subspaces $\E_B$ {\it live} in different vector spaces $\RR^{\Pi_B}$. Our next goal is to embed them into the same vector
space $\RR^{\Pi_n}$ and identify nontrivial intersection patterns in these embeddings. To compute the dimension of the choice polytope, we therefore need to understand these intersections.  For $B \subseteq [n]$, define the isomorphic copy of $\E_B$ in $\RR^{\Pi_n}$ via
$$ \widetilde{\E}_B= \Lift_B^{[n]}( \E_B) \subseteq \RR^{\Pi_n}.$$

Note that $ \widetilde{\E}_B$ is spanned by the functions $ \widetilde{ f}_B^a$ with $a \in B$ and the map sending $g \in \E_B$ to $ \widetilde{ g}:= \Lift_B^{[n]}(g) \in \widetilde{ E}_B$ defines an isomorphism of vector spaces. We can now introduce a filtration of the vector space $ \RR^{\Pi_n}$ as follows:

\begin{equation}\label{def:vk}
\V_{\le k} = \sum_{\substack{B\subseteq[n]\\1\le |B|\le k}} \widetilde{\E}_B, \qquad 1 \le k \le n.
\end{equation}

Note that
\begin{equation}\label{filt}
\V_{\le 1} \subseteq \V_{\le 2} \subseteq \cdots \subseteq \V_{\le n}. 
\end{equation}
Since $\widetilde{\E}_{\{a\}}=\RR\mathbf{1}$ for every $a\in[n]$, we have $\dim \V_{\le1}=1$.

\begin{proposition}\label{prop:choice-filtration}
For every $2\le k\le n$,
$$ \V_{\le k}/ \V_{\le k-1} \cong \bigoplus_{\substack{B\subseteq[n]\\|B|=k}} \E_B/(\E_B\cap \L_B). $$
In particular, $$ \dim \V_{\le k}-\dim \V_{\le k-1} = (k-1)\binom{n}{k}. $$
\end{proposition}

\begin{remark}
Before starting with the proof, let us elaborate on the statement of the theorem.  For each $k$-element subset $B\subseteq[n]$, consider the natural map from $\widetilde{ \E}_B$ to the quotient $\V_{\le k}/ \V_{\le k-1}$. By definition of $\V_{\le k}$, these maps together induce a surjective linear map
$$ \Phi: \bigoplus_{\substack{B\subseteq[n]\\|B|=k}} \E_B \longrightarrow \V_{\le k}/ \V_{\le k-1}, \qquad (g_B)_B\longmapsto
\sum_{|B|=k} \widetilde{ g}_B+\V_{\le k-1}. $$
The main claim of this theorem is that the kernel of $\Phi$ is given by $$ \ker\Phi = \bigoplus_{\substack{B\subseteq[n]\\|B|=k}} (\E_B\cap \L_B). $$
In other words, if $\Phi((g_B)_B)=0$ in the quotient space $\V_{\le k}/ \V_{\le k-1}$ then $g_B \in \E_B \cap \L_B$ for every $k$-element subset of $ [n]$. 
\end{remark}

\begin{proof}[Proof of Theorem \ref{prop:choice-filtration}]
In order to prove this, suppose first that $g_B\in \E_B$ are such that 
$$ \sum_{\substack{B\subseteq[n]\\|B|=k}}g_B \in V_{\le k-1}.$$
We must prove that, for every $k$-element subset $A\subseteq[n]$, we have $ g_A\in \E_A\cap \L_A. $

Fix such a set $A$. Choose once and for all a ranking $\rho$ of $A^c$. For every ranking $\pi\in\Pi_A$, let $\iota_A(\pi)\in\Pi_n$ be defined by $ \iota(\pi)= \pi \bullet \rho$. In other words, we define $\iota(\pi)$ to be the ranking obtained by first listing the elements of $A$ in the order prescribed by $\pi$, and then listing the elements of $A^c$ in the fixed order $\rho$. Thus, in every ranking in the image of $\iota_A$, all elements of $A$ precede all elements of $A^c$, while the relative order inside $A$ is allowed to vary as $\pi$ varies in $\Pi_A$.  Define the compression map 
$$ \Comp_A:\RR^{\Pi_n}\longrightarrow\RR^{\Pi_A} , \qquad \Comp_Af=f\circ\iota_A. $$
This map is indeed closely related to the lift maps defined in Section \ref{sec:lift}. We identify functions in $\E_A$ with the corresponding functions on $\Pi_A$. With this identification, $ \Comp_Ag_A=g_A$,  since $g_A$ depends only on the best element of $A$.

We next show that, for every $k$-element subset $B\neq A$, we have $  \Comp_Ag_B\in \L_A. $ Since $|A|=|B|=k$ and $A\neq B$, the intersection $B\cap A$ is a proper subset of $A$.

Assume first that $B\cap A\neq\varnothing$. In every ranking in the image of $\iota_A$, all elements of $B\cap A$ precede all elements of $B\cap A^c$. Hence the best  element of $B$ is exactly the best element of $B\cap A$. Since $g_B\in \E_B$, the function $g_B$ depends only on the most preferred element of $B$. It follows that $\Comp_Ag_B$ depends only on the ranking induced on $B\cap A$. Therefore
$$  \Comp_Ag_B\in \U_{B\cap A}\subseteq \L_A. $$
If $B\cap A=\varnothing$, then $B\subseteq A^c$. The order on $A^c$ has been fixed, so the best element of $B$ is fixed as $\pi\in\Pi_A$ varies. Hence $ \Comp_Ag_B$ is constant, and therefore again belongs to $\L_A$.

Thus $ \Comp_Ag_B\in \L_A $ for every $B\neq A$ with $|B|=k$. We also claim that
$$  \Comp_A(\V_{\le k-1})\subseteq \L_A. $$
It is enough to verify this for a function $g_C\in \E_C$ with $|C|\le k-1$. If $C\cap A\neq\varnothing$, then, for rankings in the image of $\iota_A$, every element of $C\cap A$ precedes every element of $C\cap A^c$. Hence the best element of $C$ is the best element of $C\cap A$, and therefore $\Comp_Ag_C$ depends only on the induced ranking on $C\cap A$. Since
$ |C\cap A|\le |C|\le k-1,$ the set $C\cap A$ is a proper subset of $A$, and thus $ \Comp_A g_C\in \U_{C\cap A}\subseteq \L_A. $
If $C\cap A=\varnothing$, then $ \Comp_A g_C$ is constant because the order on $A^c$ is fixed. Hence again $ \Comp_A g_C\in \L_A$. This proves $  \Comp_A(V_{\le k-1})\subseteq \L_A. $
Now suppose $$v:= \sum_{\substack{B\subseteq[n]\\|B|=k}}g_B \in V_{\le k-1}.$$ Applying $ \Comp_A$ gives
$$  \Comp_A v = g_A+ \sum_{\substack{|B|=k\\B\neq A}}  \Comp_A g_B. $$
Combining  the preceding observations and $  \Comp_A  v\in \L_A $ we deduce
$$ g_A =   \Comp_A v- \sum_{\substack{|B|=k\\B\neq A}}  \Comp_A g_B \in \L_A. $$
Since $g_A\in \E_A$ by assumption, we conclude that $ g_A\in \E_A\cap \L_A. $ As $A$ was arbitrary, this holds for every $k$-element subset $A\subseteq[n]$. Hence
$$ \ker\Phi \subseteq \bigoplus_{\substack{B\subseteq[n]\\|B|=k}} (\E_B\cap \L_B). $$

For the reverse inclusion, let $g_B\in \E_B\cap \L_B$. By Lemma~\ref{lem:choice-local}, $g_B$ is constant. Therefore $g_B\in V_{\le1}\subseteq V_{\le k-1}$. Hence every element of $$ \bigoplus_{\substack{B\subseteq[n]\\|B|=k}} (\E_B\cap \L_B) $$ lies in $\ker\Phi$. Thus
$$ \ker\Phi = \bigoplus_{\substack{B\subseteq[n]\\|B|=k}} (\E_B\cap \L_B).$$

Finally, by Lemma~\ref{lem:choice-local}, we have $ \dim E_B/(E_B\cap L_B)=k-1 $
for every $k$-element set $B$. Since there are $\binom{n}{k}$ such subsets, we obtain
$ \dim V_{\le k}-\dim V_{\le k-1} = (k-1)\binom{n}{k}. $
\end{proof}

We can now compute the dimension of the choice polytope.

\begin{theorem}\label{thm:choice-dimension}
The dimension of the choice polytope on $n$ alternatives is
$$ \dim \B_n = \sum_{k=2}^n(k-1)\binom{n}{k} = (n-2)2^{n-1}+1. $$
\end{theorem}

\begin{proof}
By Proposition~\ref{prop:choice-filtration} and the simple fact that $\dim \V_1=1$, we have 
$$ \dim V_{\le n} = 1+\sum_{k=2}^n(k-1)\binom{n}{k}. $$
The space $V_{\le n}$ is the row space of the matrix whose columns are the deterministic choice vectors indexed by rankings in $\Pi_n$. Moreover, the constant function belongs to this row space, since for every nonempty $B$,
$$ \sum_{a\in B}f^B_a=\mathbf{1}.$$
It follows that the affine dimension of the convex hull of the deterministic choice vectors is one less than the rank of this matrix. Hence $$ \dim \mathcal{C}_n = \sum_{k=2}^n(k-1)\binom{n}{k}=  (n-2)2^{n-1}+1. $$
\end{proof}

\section{Proof of Theorem \ref{thm:main}} Let us start by setting some notation and then explain the main differences between the proofs of Theorem \ref{thm:main} and
Theorem \ref{thm:choice-dimension}. In this section will use $\E_B$ and $ \widetilde{ E}_B$ to denote vector spaces defined in the best-choice setting. We hope that this overlap in notation does not confuse the reader. 

For a set $B$ with $k:=|B|\geq2$, let $$ \E_B = \operatorname{span}\{f^B_{a,b}:a,b\in B,\ a\neq b\} \subseteq \RR^{\Pi_B}, $$
where as in the previous section we have 
$$ f^B_{a,b}(\pi)=\mathbf{1}_{\{a\text{ and } b \text{ are the best and worst elements of } B\text{ under }\pi\} }= \begin{cases} 1 & \textrm{if } \, \pi_1=a, \pi_k=b\\
0 & \,  \textrm{otherwise } 
\end{cases}     $$

Recall that for every $h \in \RR^{ B^{(2)}}$, we define $\ch(h) \in \RR^\Pi$ by  $\ch(\pi)= h( \pi_1, \pi_2)$. 
Let $\{ \e_{a,b} \}_{ (a, b) \in B^{(2)}}$ be the canonical basis of $\RR^{ B^{(2)}}$. It is easy to see that  $f_{a,b}^B= \ch{\e_{a,b}}$
This implies that the space of functions on the rankings determined by the best and the worst choice is given by  
$$\E_B= \ch(\RR^{ B^{(2)}}) \subseteq  \RR^{\Pi_{ B }}.$$ One can verify that  the $|B|(|B|-1)$ functions $f^B_{a,b}$ are linearly independent. Hence
$$\dim \E_B=|B|(|B|-1).$$

We will now turn to the proof of Theorem \ref{thm:main}. The proof has the same structure, but the technical details are more involved. 
One essential difference between the proof involves the intersection structure of $\E_B\cap \L_B$. In the ordinary choice model, using the possible best marginals that 
can arise from elements in $L_B^\perp$ (Lemma~\ref{lem:choice-first-marginal}) we deduce that the only element of $\E_B$ coming from proper-subset of $B$ (that is $\E_B \cap \L_B)$ are constant. This information was finally assembled to compute the dimension of $\V_n$. 

As we shall see the possible best-worst marginals coming from $\L_B^\perp$ has a more complicated structure (see Lemma \ref{lem:endpoint-image}). The additional constraints on these marginals, will reduce the intersection dimensions (Lemma \ref{prop:local}). 
The filtration argument itself is essentially unchanged when passing from ordinary choice to best--worst choice. We will elaborate on some (minor) differences in the course of the proof.

\medskip

In order to avoid too many indices, for a ranking $\pi$ of $B$ and a nonempty set $C\subseteq B$, write $\operatorname{best}_\pi(C)$ and $\operatorname{worst}_\pi(C)$ for the best and worst elements of $C$ in the ranking of $C$ induced by $\pi$.  More precisely, we set $$\operatorname{best}_\pi(C)= (\pi|_C)_1, \quad \text{ and } \quad  \operatorname{worst}_\pi(C):= (\pi|_C)_{|C|}.$$

The next lemma must be contrasted with Lemma \ref{lem:choice-local}. It shows that contrary to the best case, the intersection $\L_B \cap \E_B$ contains non-constant functions. 

\begin{lemma}\label{lem:alternating}
Let $B$ be a set of cardinality $k\geq2$ and let $u\in\RR^B$. For  $\pi=(\pi_1,\ldots,\pi_k)\in\Pi_B$, we have
$$ (\chb+ (-1)^k \chw)(u)= u_{\pi_1}+(-1)^k u_{\pi_k} = \sum_{\substack{C\subseteq B\\ 1 \le |C| \le k-1}} (-1)^{|C|+1} u_{\operatorname{worst}_\pi(C)}. $$
In particular, the image of $\RR^B$ under the linear map $\chb+ (-1)^k \chw$ is included in $\L_B$.
\end{lemma}

\begin{proof}
Write $\pi=(x_1,\ldots,x_k)$. For a fixed $1 \le r \le k$, the term $u_{x_r}$ occurs in the sum over $j$-element subsets precisely when $x_r$ is the worst element of the chosen subset. Such a subset is obtained by choosing $j-1$ elements from $\{x_1,\ldots,x_{r-1}\}$ and adjoining $x_r$. Hence the coefficient of $u_{x_r}$ on the right-hand side is
$\sum_{j=1}^{k-1}(-1)^j\binom{r-1}{j-1}$.

For $r=1$ this coefficient equals $1$. For $2\leq r\leq k-1$, it is $(1-1)^{r-1}=0$. Finally, for $r=k$ it equals
$-(-1)^{k-1}=(-1)^k$. This proves the identity. Every set $C$ occurring on the right has at most $k-1$ elements, and  is  hence a proper subset of $B$, which proves the last assertion.
\end{proof}

\begin{remark}
 Recall the definition of the marginal operators given in \eqref{maginaldef}
\[ \mar: \RR^{\Pi_B} \to \RR^{B^{(2)}} , \quad  \mar(h)(a,b)= \sum_{ \pi_1= a, \pi_k=b} h(\pi) \]
The {\it partial sum operators} $ \psum_1, \psum_2: \RR^{ B^{(2)}} \to \RR^B$ are defined by 
\[ \psum_1( u)(a)= \sum_{y \in B \setminus \{ a \} } u(a,y), \qquad \psum_2( u)(a)= \sum_{x \in B \setminus \{ b \} } u(x,b). \]
Recall from Remark \ref{psums} that $\mar(h)$ can be viewed as a $|B| \times |B|$ matrix. When regarded in this way, the 
commutativity of the diagram below corresponds to the fact that $\marb(h)$ and $\marw(h)$ are simply given as the row sum and column sum of this matrix.  
\end{remark}

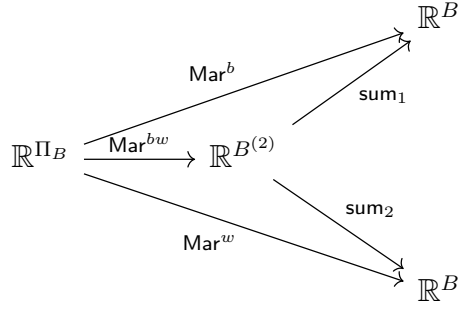
\[
\begin{tikzcd}[column sep=large, row sep=large]
& & \RR^B \\
\RR^{\Pi_B}
\arrow[r, "\mar"]
\arrow[urr, "\marb"]
\arrow[drr, "\marw"']
& \RR^{B^{(2)}}
\arrow[ur, "\psum_{1}"']
\arrow[dr, "\psum_2"]
& \\
& & \RR^B
\end{tikzcd}
\]

\begin{lemma}\label{lemma:zsum}
Let $B$ be a set with $k=|B| \ge 4$. Let $ \{ \e_{a,b}: \, (a, b) \in B^{(2)} \}$ denote the canonical basis for $\RR^{ B^{(2)}}$. Let $W$ be the subspace 
of $ \RR^{ B^{(2)}}$ consisting of vectors $v$ satisfying $\psum_1(v)= \psum_2(v)=0$. Then $W$ is equal to the span of the set of vectors
$$ \k_{a,b,c,d }:= (\e_{a,c}+ \e_{b,d}) - \left(  \e_{a,d} + \e_{b,c} \right),  $$
where $a,b, c, d$ are pairwise distinct elements of $B$. Moreover, we have $\dim W=k^2-3k+1$. 
\end{lemma}

\begin{proof}
Denote by $W_0$ the subspace spanned by all $\k_{a,b,c,d }$ where $a,b, c, d$ are pairwise distinct elements of $B$.
First, note that $\psum_1 (\e_{x,y})=\e_x$ and $\psum_2(e_{x,y})= \e_y$  for all $(x, y) \in B^{(2)}$. This implies  $\psum_1(k_{a,b,c,d })= \psum_2(k_{a,b,c,d})= 0$. 
This implies that $W_0 \subseteq W$. In order to prove the reverse inclusion, consider the orthogonal complement $W_1$ of $W_0$ inside $W$. In other words,
suppose $v \in W$ is orthogonal to all elements of $W_0$. We show that  $v=0$. Since $ \langle v, \e_{x,y} \rangle= v(x,y)$, the orthogonality condition implies
$$ v(a,c)- v(b,c) =v(a,d) - v(b,d). $$
whenever $a,b,c,d$ are pairwise distinct. For fixed distinct $a,b$, this implies that the function $x \mapsto v(a,x)- v(b,x)$ 
is constant on $ B \setminus \{a,b\}$; denote this constant by $d(a,b)$. If $a,b,c$ are distinct, choose $j$ outside $\{a,b,c\}$, which is possible because $k\geq4$. This implies that  
\[ d(a,b)+d(b,c)= v(a,j)- v(b, j) + v(b, j)- v(c,j)= v(a, j)- v(c, j)= d(a, c). \]
Also, 
$$d(a, c)= v(a, j)- v(c,j)= - (v (c, j)- v( a, j) )= -d (c,a).$$

We claim that these two properties of $d$ imply that  there exists a function  $\alpha: B \to \RR$ such that $d(a,b)=\alpha(a)-\alpha(b)$ for distinct $a,b$ in $B$. In fact, fix 
$c \in B$ and for $a \in B \setminus \{ c \}$ define $\alpha(a)= d(a, c)$ and extend $\alpha$ to $B$ by setting $\alpha(c)=0$. For $a, b$ both distinct from $c$, we have 
\[ d(a, b)= d(a,c)- d(b, c)= \alpha(a)- \alpha(b). \]
Finally, since $\alpha(c)=0$, we have $d(a, c)= \alpha(a)- \alpha(c)$ and $d(c,b)= - d(b,c)= \alpha(c)- \alpha(b)$.
Putting these together, we deduce that for a fixed $b \in B$, and $a, a' \in B \setminus \{ b \}$, we have
\[ v(a, b)- v(a',b)= d(a, a')= \alpha(a)-\alpha(a') \hspace{2mm} \Rightarrow \hspace{3mm}  v(a, b)- \alpha(a)= v(a',b)-\alpha(a'). \]
Hence, the function $x \mapsto v(x, b)- \alpha(x)$ is constant on $ B \setminus \{b \}$. Denote its values by $\beta(b)$. Hence, we have
$v(a, b)= \alpha(a)+ \beta(b)$ for all $(a, b) \in B^{(2)}$.   
Now, we have for every $ a \in B$, 
\[ 0= \psum_1( v)(a)= \sum_{ y \neq a} \alpha(a)+ \beta(y)= (|B|-1) \alpha(a)- \beta(a)+  \sum_{y \in B} \beta(y). \]
Fixing $a$ and letting $b$ vary, we deduce that the function $x \mapsto (|B|-1) \alpha(x)- \beta(x)$ is a constant function.  Similarly, 
the function $x \mapsto (|B|-1) \beta(x)- \alpha(x)$ is also constant. Since $|B|>2$, we deduce that $\alpha$ and $\beta$ are both constant. Denote the
common value of these two functions by $A_1$ and $A_2$, we deduce that $ 0= (|B|-1) (A_1+A_2)$. Hence $v(x, y)= A_1+A_2=0$. This proves the claim. 

The dimension count is straightforward and is best seen if we view $v$ as a $k$ by $k$ matrix with zero diagonal entries with the constrains that row sums
and column sums are all zero. There are $k^2-k$ entries in $v$. Since each row must add up to $0$, we obtain $k$ additional constraints. Similarly, 
we have $k$ constraints coming from the column sums. Both sets of constraints imply that all matrix entries add up to $0$. Hence $ \psum_1(v)= \psum_2(v)=$ imposes a total of $2k-1$ linear constraints on $v$. Hence 
$\dim W= (k^2-k)-(2k-1)= k^2-3k+1$. 
\end{proof}

The next lemma is an analog of Lemma \ref{lem:choice-first-marginal} for the best-worst choice polytope. 

\begin{lemma}\label{lem:endpoint-image}
Let $B$ have cardinality $k\geq4$. Then
\begin{equation}\label{marginalimage}
\mar (\L_B^\perp) = \RR^{ B^{(2)}}_0 \cap \left\{ v \in \RR^{ B^{(2)}}:  \psum_1(v)+ (-1)^{k} \psum_2(v)=0 \right\}. 
\end{equation}
Moreover, we have $ \dim \mar (\L_B^\perp)=k(k-2).$
\end{lemma}

\begin{proof}
First note that for every $h \in \L_B^\perp$ we have 
\[ \langle \mar(h), \one \rangle_{ B^{(2)}} = \langle h, \ch(\one) \rangle_{\Pi_B}= \langle  h, 1 \rangle_{\Pi_B} =0, \]
where the latter follows from $\one \in \L_B$. This shows that $  \mar (\L_B^\perp) \subseteq \RR^{ B^{(2)}}_0$.
To prove the second inclusion, suppose $v= \mar(h)$ for some $h \in \L_B^\perp$. Then using Lemma \ref{lem:alternating}
 for every $u \in \RR^B$ we have 
\begin{equation}
\begin{split}      
\langle \psum_1 (v) + (-1)^k \psum_2( v), u \rangle &= \langle  \psum_1 (\mar(h)) + (-1)^k \psum_2( \mar(h)), u  \rangle \\
&= \langle  \marb(h), u  \rangle + \langle \marw(h),  u \rangle \\
&= \langle h, (\chb+ (-1)^k \chw)(u) \rangle =0
\end{split}
\end{equation}
We have thus established the inclusion of $\mar (\L_B^\perp)$ in the right-hand side. We will now prove the reverse inclusion. 
Set   
$$ \Z_B = \{h \in \RR^{ B^{(2)}} : \psum_1(h)=0,  \, \psum_2(h)=0\}. $$

We claim that $\Z_B\subseteq  \mar (\L_B^\perp) $. Take four distinct elements $a,b,c,d\in B$, and choose a partition $B=A\sqcup C$ with $a,b\in A$, $c,d\in C$, and $|A|,|C|\geq2$. Using Lemma \ref{lem:choice-first-marginal}, there exist
$h_A\in \L_A^\perp$ with $\marb(h_A)= \e_a - \e_b$. Similarly, there exists  $h_C \in \L_C^\perp$ such that 
$\marw(h_C)=\e_c- \e_d$. 
Define $h$ on $\Pi_B$ first by setting for every $ \sigma \in \Pi_A$ and $\tau \in \Pi_C$ 
$$ h(\sigma \bullet \tau)=h_A(\sigma)h_C(\tau) $$
and setting it to be zero on the rest of $\Pi_B$.  We claim that $h$ defined as above belongs to $\L_B^\perp$. In order to show this, we need to prove for every $x \in B$ and every $ \tau \in \Pi_{ B \setminus \{ x \}}$  we have 
\[ \sum_{ \pi: \pi \setminus x = \tau } h( \pi) =0. \]

First, assume that $x\in A$. If the remaining ranking does not have the block $A\setminus\{x\}$ before the block $C$, every insertion has value zero; otherwise the nonzero part of the insertion sum is the value of $h_C$ on the fixed $C$-block multiplied by an insertion sum for $h_A$. The case of deleting an element of $C$ is symmetric. Hence $h\in \L_B^\perp$. This construction gives
$$ \mar(h)=  \e_{a,c}+ \e_{b,d} - \left(  \e_{a,d} + \e_{b,c} \right)= \k_{a,b,c,d }. $$
It follows that  $\mar(\L_B^\perp)$ contains every $\k_{a,b,c,d }$ with $a,b,c,d$ pairwise distinct elements of $B$.  Lemma \ref{lemma:zsum} implies that 
$\mar(\L_B^\perp)$  contains $\Z_B$. 
To finish the proof, let $v$ belong to the right-hand side of \eqref{marginalimage}. We need to show the existence of $ h \in \L_B^\perp$ with 
$\mar(h)=v$.  Since $$ \langle \psum_1 (v), \one \rangle = \langle v, \one \rangle =0,$$ by Lemma \ref{lem:choice-first-marginal}, there exists $h_0\in \L_B^\perp$ with $\marb(h_0)= \psum_1(v)$. This implies that 
\[ \psum_1( v - \mar (h_0) ) = \psum_1(v)- \marb(h_0)=0. \]
Since $ v - \mar (h_0)$ satisfies the equation $\psum_1(v)+ (-1)^{k} \psum_2(v)=0$, we deduce that 
$$ \psum_2( v - \mar (h_0) ) = (-1)^{k+1}  \psum_1( v - \mar (h_0) ) =0.$$
It follows that $v - \mar (h_0) \in \Z_B$. By the above discussion, there exists $h_1 \in \L_B^\perp$ with $\mar(h_1)= v - \mar (h_0) $. We deduce
$ v= \mar ( h_0+ h_1)$, finishing the first claim. Note that we have proven that the kernel of the map $\psum_1$ from $\mar( \L_B^\perp)$ to $\RR^B$ has kernel
$\Z_B$ and image $\RR^B_0$. Hence, 
\[ \dim \mar( \L_B^\perp)= \dim \Z_B+ \dim \RR^B_0= (k^2-3k+1)+(k-1) = k(k-2). \]
\end{proof}

The next proposition is the analog of Proposition \ref{lem:choice-local}. 

\begin{proposition}\label{prop:local}
Suppose $k=|B|\geq4$. Then
$$
\dim \E_B/( \E_B\cap \L_B)=k(k-2).
$$
More precisely, a function $g\in \E_B$ belongs to $\E_B\cap \L_B$ if and only if it has one of the following forms:
\begin{enumerate}
\item if $k$ is even,
$g(\pi)=\alpha_{\pi_1}+\alpha_{\pi_k}$ for some $\alpha \in\RR^B$;
\item if $k$ is odd,
$g(\pi)=\lambda+\alpha_{\pi_1}-\alpha_{\pi_k}$ for some $\lambda\in\RR$ and $\alpha \in\RR^B$.
\end{enumerate}
Moreover, every function in $\E_B\cap \L_B$ belongs to $ \sum_{C\subsetneq B} \E_C$.
\end{proposition}

\begin{proof}
It is clear that $ \RR \one \subseteq \E_B \cap \L_B$. To prove the reverse inclusion, let $g\in \E_B \cap \L_B$. Since $g \in \E_B$, we can write  $g = \ch(u)$ for some $u \in \RR^B$. 
 Since $g \in \L_B$, we know that $g$ is orthogonal to every element of $\L_B^\perp$. This implies that for every $h \in \L_B^\perp$ we have
 \[ 0 = \langle g, h \rangle_B=   \langle \ch(u), h \rangle_B = \langle u, \mar(h) \rangle_{\Pi_B}. \]
 For notational convenience, write $W_B= \mar ( \L_B^\perp)$. Hence, it follows that $\E_B \cap L_B= \ch( W_B)$. We will give a precise description of $W_B$, which will lead to a description of $\E_B \cap L_B$. 

\medskip

\underline{Case 1: $k$ even}. Suppose first that $k$ is even. Let $\alpha \in \RR^B$ and define $u \in \RR^{ B^{(2)}}$  by $u(a,b) = \alpha(a)+ \alpha(b)$. We claim that $ u$ is orthogonal to 
$\mar ( \L_B^\perp)$. Let $v \in \mar ( \L_B^\perp)$. Then we have 
\begin{equation}
\begin{split}      
 \langle  u, v \rangle  & = \sum_{ (a, b) \in B^{(2)}} v(a, b) ( \alpha(a)+ \alpha (b) )= \sum_a \alpha(a) \sum_{b \neq a} v(a, b)+ \sum_b  \alpha(b) \sum_{a \neq b} v(a,b) \\
&= \sum_a \alpha(a) \psum_1(v) +  \sum_b \alpha(b) \psum_2(v) =  \sum_a \alpha(a) (\psum_1(v) + \psum_2(v) )=0. 
\end{split}
\end{equation}
This gives an injective map from $\RR^B$ to the orthogonal complement of $\mar ( \L_B^\perp)$ in $ \RR^{ B^{(2)}}$. Since this orthogonal complement has dimension 
$k$, this map is surjective. This shows that $ W_B= \{ u(a,b):= \alpha(a)+ \alpha(b): \alpha \in \RR^B \}$. Hence
\[ \E_B \cap \L_B= \ch( W_B)= \{ \alpha_{\pi_1} + \alpha_{\pi_k}:  \alpha \in \RR^B \}. \]
\medskip

\underline{Case 2: $k$ odd}.
Let $\alpha \in \RR^B$ and define $u \in \RR^{ B^{(2)}}$  by $u(a,b) = \alpha(a)- \alpha(b)$. A similar computation shows that $u$ belongs to the orthogonal complement of 
$\mar ( \L_B^\perp)$. The image of these functions determines a $(k-1)$-dimensional subspace of $\RR^{ B^{(2)}}$, as the constant function in $\RR^B$ maps to zero. Together with the constant function $\one$, we obtain a $k$-dimensional subspace orthogonal to $\mar ( \L_B^\perp)$. This finishes the proof in this case.

Finally we will show that the  intersection is generated by lower-order best--worst coordinates. Lemma~\ref{lem:alternating} expresses
$u_{\pi_1}+(-1)^k u_{\pi_k}$ as a linear combination of the functions
$u_{\operatorname{worst}_\pi(C)}$ with $C\subsetneq B$. If $|C|\geq2$, then
$$ u_{\operatorname{worst}_\pi(C)} = \sum_{b\in C}u(b) \sum_{a\in C\setminus\{b\}}f^C_{a,b}(\pi). $$
The contribution from one-element sets in Lemma~\ref{lem:alternating} is constant, and a constant function is already in the span of the two best--worst coordinates associated with any two-element set. The explicit descriptions above therefore show that every element of $\E_B\cap \L_B$ is generated by best--worst coordinates on proper subsets of $B$.
\end{proof}

\begin{remark}\label{impo}
Recall that $\L_B$ is, by definition,  the vector spanned by functions of {\it rankings} of all  proper subsets of $B$. By definition,
$$ \L_B=\sum_{C\subsetneq B} \U_C, $$
On the other hand, the subspace
$$ \sum_{C\subsetneq B} \E_C, $$
consists of functions that are linear combination of {\it best-worst coordinates} of proper subsets of $B$ and is, in general, a proper subspace of $\L_B$. Thus the assertion that every element of $\E_B\cap \L_B$ belongs to this subspace is stronger than the inclusion $E_B\cap L_B\subseteq L_B$. This stronger statement will be necessary later. 
\end{remark}

The case of three alternatives is exceptional.

\begin{lemma}\label{lem:three}
If $|B|=3$, then
$$ \dim \E_B/(\E_B\cap \L_B)=2. $$
Moreover, $\E_B\cap \L_B$ is generated by the best--worst coordinates on the two-element subsets of $B$.
\end{lemma}

\begin{proof}
For three alternatives, the best and worst elements determine the entire ranking. Hence
$\E_B=\RR^{\Pi_B}$ and $\dim \E_B=6$. On the other hand, $\L_B$ is spanned by the constant function and the three pairwise comparison functions. These four functions are linearly independent, so $\dim L_B=4$. Since $\E_B=\RR^{\Pi_B}$, one has $\E_B\cap L_B=L_B$, and the quotient has dimension $2$. Each pairwise comparison function is one of the two deterministic best--worst coordinates on the corresponding two-element set, while the constant is their sum.
\end{proof}

As in the case of best choice, we have a linear isomorphism $ g \mapsto \widetilde{ g}=\Lift_{B}^{[n]}(g)$ from $\RR^{\Pi_B}$ to $\RR^{\Pi_n}$. 
For $B \subseteq [ n]$, define the isomorphic copy of $\E_B$ in $\RR^{\Pi_n}$ via
$$ \widetilde{\E}_B= \Lift_B^{[n]}( \E_B) \subseteq \RR^{\Pi_n}.$$

Note that $ \widetilde{\E}_B$ is spanned by the functions $ \widetilde{ f}_{a,b}^B$ with $(a,b) \in B^{(2)}$ and the map sending $g \in \E_B$ to $ \widetilde{ g}:= \Lift_B^{[n]}(g) \in \widetilde{ E}_B$ defines an isomorphism of vector spaces. We can now introduce a filtration of the vector space $ \RR^{\Pi_n}$ as follows:
We regard $\E_B$ as a subspace of $\RR^{\Pi_n}$ by composing with the restriction map $\pi\mapsto\pi|_B$.  For $1\le k\le n$, define
$$ \V_{\le k} = \RR\one+ \sum_{\substack{B\subseteq[n]\\2\le |B|\le k}}\E_B. $$
Thus
$$ \V_{\le1}\subseteq\V_{\le2}\subseteq\cdots\subseteq\V_{\le n}, $$
where $\V_{\le1}=\RR\one$. Notice that $\one\in\E_B$ whenever $|B|\ge2$, since
$$ \sum_{(a,b)\in B^{(2)}}f^B_{a,b}=\one. $$

\begin{proposition}\label{prop:bw-filtration}
For every $2\le k\le n$,
$$ \V_{\le k}/\V_{\le k-1} \cong \bigoplus_{\substack{B\subseteq[n]\\|B|=k}}\E_B/(\E_B\cap\L_B). $$
Consequently,
$$ \dim\V_{\le k}-\dim\V_{\le k-1} =
\begin{cases}
\binom{n}{2}, & k=2,\\[2mm]
2\binom{n}{3}, & k=3,\\[2mm]
k(k-2)\binom{n}{k}, & k\ge4.
\end{cases} $$
\end{proposition}

\begin{remark}
As in Proposition~\ref{prop:choice-filtration}, for every $k$-element subset $B\subseteq[n]$ there is a natural map from $\E_B$ to the quotient $\V_{\le k}/\V_{\le k-1}$. Together these maps induce a surjective linear map
$$
\Phi: \bigoplus_{\substack{B\subseteq[n]\\|B|=k}}\E_B \longrightarrow \V_{\le k}/\V_{\le k-1}, \qquad (g_B)_B\longmapsto \sum_{|B|=k}\widetilde{ g}_B+\V_{\le k-1}.
$$
The main point is to prove that
$$ \ker\Phi = \bigoplus_{\substack{B\subseteq[n]\\|B|=k}} (\E_B\cap\L_B). $$
Equivalently, if
$$ \sum_{|B|=k}g_B\in\V_{\le k-1}, $$ then each individual $g_B$ must belong to $\E_B\cap\L_B$.
\end{remark}

\begin{proof}
Suppose first that $g_B\in\E_B$, for $|B|=k$, satisfy
$$ \sum_{\substack{B\subseteq[n]\\|B|=k}}g_B \in\V_{\le k-1}. $$
We will show that $$ g_A\in\E_A\cap\L_A $$ for every $k$-element subset $A\subseteq[n]$.

Fix such an $A$. Choose once and for all a ranking $\rho\in\Pi_{A^c}$. For every $\pi\in\Pi_A$, define
$$ \iota_A(\pi)=\pi\bullet\rho. $$ Thus every element of $A$ occurs before every element of $A^c$, the relative order on $A$ is allowed to vary, and the relative order on $A^c$ is fixed.
As in the best-only case, define the compression map
$$ \Comp_A:\RR^{\Pi_n}\longrightarrow\RR^{\Pi_A}, \qquad \Comp_A(f)=f\circ\iota_A. $$
Thus we have  $$ \Comp_A \widetilde{ g}_A=g_A. $$
The reason for this is that  $g_A$ depends only on the best and worst elements of $A$, and the ranking induced on $A$ by $\iota_A(\pi)$ is precisely $\pi$.

We next analyze $\Comp_Ag_B$ when $B\neq A$ and $|B|=k$. Put
$$ D=A\cap B, \qquad F=B\cap A^c. $$
Since $|A|=|B|=k$ and $A\neq B$, we have $F\neq\varnothing$ and $D\subsetneq A$.

Assume first that $D\neq\varnothing$. Since every element of $A$ precedes every element of $A^c$ in the ranking $\iota_A(\pi)$, we have
$$ \best_B(\iota_A(\pi)) = \best_D(\pi). $$
On the other hand, the worst element of $B$ necessarily belongs to $F$, and since the ranking of $A^c$ is fixed, it is independent of $\pi$. More precisely, if
we set $b_{B,\rho}:=\worst_F(\rho),$ then $$ \worst_B(\iota_A(\pi))=b_{B,\rho}$$ for every $\pi\in\Pi_A$.

Write
$ g_B=\ch(u_B) $ for some $u_B\in\RR^{B^{(2)}}$. It follows that
$$ (\Comp_Ag_B)(\pi) = u_B\bigl(\best_D(\pi),b_{B,\rho}\bigr). $$
Define $v_{B,A}\in\RR^D$ by $$ v_{B,A}(a)=u_B(a,b_{B,\rho}). $$
Then
$$(\Comp_Ag_B)(\pi) = v_{B,A}\bigl(\best_D(\pi)\bigr). $$
Thus the compression of $g_B$ is in fact a best-only function on $D$. In particular,
$$ \Comp_Ag_B\in\U_D\subseteq\L_A. $$

If $D=\varnothing$, then $B\subseteq A^c$. Since the order on $A^c$ is fixed, both the best and worst elements of $B$ are fixed as $\pi$ varies in $\Pi_A$. Therefore $\Comp_Ag_B$ is constant, and hence again belongs to $\L_A$.

We have therefore shown that
$$ \Comp_Ag_B\in\L_A $$
for every $B\neq A$ with $|B|=k$.

We next prove that $$ \Comp_A(\V_{\le k-1})\subseteq\L_A. $$
It is enough to consider an element $g_C\in\E_C$ with $2\le |C|\le k-1$. Put
$D=C\cap A$ and  $F=C\cap A^c.$ There are three cases.

If $F=\varnothing$, then $C\subseteq A$. In this case the compression does not change the induced ranking on $C$, and hence
$$ \Comp_Ag_C=g_C\in\E_C\subseteq\U_C\subseteq\L_A, $$ since $|C|\le k-1<|A|$.

If both $D$ and $F$ are nonempty, write $g_C=\ch(u_C)$. Exactly as above,
$$ \best_C(\iota_A(\pi)) = \best_D(\pi),$$ whereas
$$ \worst_C(\iota_A(\pi)) = b_{C,\rho}, \qquad b_{C,\rho}:=\worst_F(\rho), $$
is independent of $\pi$. Consequently, 
$$ (\Comp_Ag_C)(\pi) = u_C\bigl(\best_D(\pi),b_{C,\rho}\bigr). $$
Thus $\Comp_Ag_C$ depends only on the best element of $D$, and in particular $$ \Comp_Ag_C\in\U_D\subseteq\L_A. $$

Finally, if $D=\varnothing$, then $C\subseteq A^c$, and both endpoints of the induced ranking on $C$ are fixed. Hence $\Comp_Ag_C$ is constant and belongs to $\L_A$.

Since also $\Comp_A(\one)=\one\in\L_A$, these three cases prove $ \Comp_A(\V_{\le k-1})\subseteq\L_A.$

Now set $$ v:= \sum_{\substack{B\subseteq[n]\\|B|=k}}g_B \in\V_{\le k-1}.$$
Applying $\Comp_A$, we obtain
$$ \Comp_Av = g_A+ \sum_{\substack{|B|=k\\B\neq A}}\Comp_Ag_B. $$
As
$ \Comp_Av\in\L_A $ and $\sum_{\substack{|B|=k\\B\neq A}}\Comp_Ag_B\in\L_A, $ we deduce 
$$ g_A = \Comp_Av- \sum_{\substack{|B|=k\\B\neq A}}\Comp_Ag_B \in\L_A.$$
Since $g_A\in\E_A$ by assumption, we conclude that $$ g_A\in\E_A\cap\L_A.$$
As $A$ was arbitrary, this proves
$$ \ker\Phi \subseteq \bigoplus_{\substack{B\subseteq[n]\\|B|=k}} (\E_B\cap\L_B). $$

We now prove the reverse inclusion. Recall that in  the best-only polytope, Lemma~\ref{lem:choice-local} gives
$$ \E_B\cap\L_B=\RR\one, $$ and therefore every element of $\E_B\cap\L_B$ automatically belongs to the first stage of the filtration.

For the best--worst polytope, the intersection $\E_B\cap\L_B$ is larger. We therefore use the stronger conclusion of Proposition~\ref{prop:local} (see also Remark \ref{impo}) namely
$$ \E_B\cap\L_B \subseteq \sum_{C\subsetneq B}\E_C. $$
In particular, if $|B|=k$, then
$$ \E_B\cap\L_B \subseteq \V_{\le k-1}. $$
Hence, if $g_B\in\E_B\cap\L_B$ for every $k$-element subset $B$, then
$$ \sum_{|B|=k}g_B\in\V_{\le k-1}, $$
and therefore $(g_B)_B\in\ker\Phi$. This proves
$$ \bigoplus_{\substack{B\subseteq[n]\\|B|=k}} (\E_B\cap\L_B) \subseteq \ker\Phi.$$

Combining the two inclusions gives
$$ \ker\Phi = \bigoplus_{\substack{B\subseteq[n]\\|B|=k}} (\E_B\cap\L_B). $$
from which the claim follows.

It remains only to use the local dimension calculation. For $|B|=2$,
$$ \dim \E_B/(\E_B\cap\L_B)=1.$$
For $|B|=3$,
$$ \dim \E_B/(\E_B\cap\L_B)=2. $$
Finally, for $|B|=k\ge4$,
$$ \dim \E_B/(\E_B\cap\L_B)=k(k-2). $$
Since there are $\binom{n}{k}$ subsets of $[n]$ of cardinality $k$, we obtain
$$ \dim\V_{\le k}-\dim\V_{\le k-1}=
\begin{cases}
\binom{n}{2}, & k=2,\\[2mm]
2\binom{n}{3}, & k=3,\\[2mm]
k(k-2)\binom{n}{k}, & k\ge4.
\end{cases} $$
This completes the proof.
\end{proof}

We can now finish the proof of the dimension formula.

\begin{proof}[Proof of Theorem~\ref{thm:main}]
A telescoping sum applied with Lemma \ref{prop:bw-filtration} gives
$$ \dim V_{\le n} = 1+\binom{n}{2}+2\binom{n}{3} + \sum_{k=4}^n k(k-2)\binom{n}{k}. $$

The dimension of $V_n$ is the rank of the matrix whose columns are the deterministic best--worst vectors $v_\pi$. Imposing the one normalizing condition on probability vectors gives 
 $$ \dim \BW_n = \dim V_n-1=  n(n-3)2^{n-2}+n+\binom{n}{2}-\binom{n}{3}. $$
\end{proof}

\end{document}